\documentclass[a4paper,12pt]{amsart}
\usepackage{latexsym,amscd,amssymb,amsmath,mathrsfs,comment,url}
\makeatletter
\theoremstyle{plain}
  \newtheorem{thm}{Theorem}[section]
  \newtheorem{prop}[thm]{Proposition}
  \newtheorem{lem}[thm]{Lemma}
  \newtheorem{cor}[thm]{Corollary}
  
\theoremstyle{definition}
  \newtheorem{dfn}[thm]{Definition}
  \newtheorem{exmp}[thm]{Example}

  \newtheorem{rem}[thm]{Remark}

\numberwithin{equation}{section}

\let\opn\operatorname %abbreviation of \operatorname
\newcommand\ba{\mathbf{a}}
\renewcommand\iff{\Longleftrightarrow}
\newcommand\bb{\mathbf{b}}
\newcommand\bd{\mathbf{d}}
\newcommand\NN{\mathbb{N}}

\newcommand\cF{\mathcal{F}}
\def\kk{\Bbbk}
\newcommand\wS{\widetilde{S}}
\newcommand\Ext{\opn{Ext}}
\newcommand\Ass{\opn{Ass}}
\newcommand\BoX{\opn{\mathsf{b-pol}}}
\newcommand\codim{\opn{codim}}
\newcommand\chara{\opn{char}}
\newcommand\pd{\opn{proj.\! dim}}
\newcommand\pol{\opn{\mathsf{pol}}}
\def\m{\mathsf{m}}
\def\n{\mathsf{n}}

\def\fp{\mathfrak{p}}
\def\fq{\mathfrak{q}}

\title{Alternative polarizations of Borel fixed ideals}
\author{Kohji Yanagawa}
\thanks{The author is partially supported by Grant-in-Aid for Scientific Research (c) (no.22540057).}
\address{Department of Mathematics, Kansai University,
Suita 564-8680, Japan}
\email{yanagawa@ipcku.kansai-u.ac.jp}
\begin{document}
\maketitle
\begin{abstract}
For a monomial ideal $I$ of a polynomial ring $S$,  
a {\it polarization} of $I$ is a squarefree monomial ideal $J$ of a ``larger" 
polynomial ring $\wS$ such that $S/I$ is a quotient of $\wS/J$ by 
a (linear) regular sequence. We show  that a Borel fixed ideal admits a ``non-standard" polarization. 
For example, while the usual polarization sends $xy^2 \in S$ to $x_1y_1y_2 \in \wS$, 
ours sends it to $x_1y_2y_3$. Using this idea, 
we  recover/refine the results on {\it squarefree operation} in the 
shifting theory of simplicial complexes.   The present paper generalizes a result of Nagel and Reiner, 
while our approach is very different from theirs. 
\end{abstract}
\section{introduction}
Let $S:=\kk[x_1, \ldots, x_n]$ and $\wS:=\kk[ \, x_{i,j} \mid 1 \le i \le n, \, 1 \le j \le d \, ]$ 
be polynomial rings over a field $\kk$. Any monomial $\m \in S$ has a unique expression 
\begin{equation}\label{alpha expression}
\m = \prod_{i=1}^e x_{\alpha_i} \quad \text{with} \quad 1 \le \alpha_1 \le \alpha_2 \le \cdots 
\le \alpha_e \le n. 
\end{equation}
If $\deg (\m) \, (=e) \le d$, we set
\begin{equation}\label{b-pol}
\BoX(\m)= \prod_{i=1}^e x_{\alpha_i, i} \in \wS.
\end{equation} 
Note that $\BoX(\m)$ is a squarefree monomial.  
For a monomial ideal $I \subset S$, $G(I)$ denotes the set of minimal (monomial) 
generators of $I$. If $\deg(\m) \le d$ for all $\m \in G(I)$, we set 
$$\BoX(I):= (\BoX(\m) \mid \m \in G(I) )\subset \wS.$$ 

In Theorem~\ref{main}, we will show that if $I$ is {\it Borel fixed} (i.e.,  $\m \in I$, $x_i|\m$ and $j <i$  
imply $(x_j/x_i) \cdot \m \in I$), then $J:=\BoX(I)$ is a polarization of $I$, that is, 
$\Theta := \{ x_{i,1}-x_{i,j} \mid 1 \leq i \leq n, \, 2 \leq j \leq d \, \} \subset \wS$
forms a $\wS/J$-regular sequence with the canonical isomorphism $\wS/(J+(\Theta)) \cong S/I$. 
For general monomial ideals, the corresponding statement is not true. Even for a Borel fixed ideal, 
$\BoX$ is essentially different from the standard polarization, see Example~\ref{1st exm}.   
Recall that Borel fixed ideals play an important role in Gr\"obner basis theory and many related areas, 
since they appear as the {\it generic initial ideals} of homogeneous ideals (c.f. \cite[\S15.9]{E}). 

The idea of $\BoX(I)$ first appeared in Nagel and Reiner \cite{NR}, while they did not give a specific name to this 
construction. Among other things, under the additional assumption that all elements of $G(I)$ have the same degree, 
they have shown the above result (it is not directly stated there, but follows from 
\cite[Theorem~3.13]{NR}).  Inspired by this, Lohne~\cite{L} undertakes a study 
of all possible polarizations of certain monomial ideals.
He calls $\BoX(I)$ the {\it box polarization}, 
since combinatorial objects consisting of  ``boxes" are used in \cite{NR}. 
While the name ``box" is no longer natural in our case, we use the symbol $\BoX$.

To prove Theorem~\ref{main}, we show that $\wS/J$ has a {\it pretty clean filtration} introduced by 
Herzog and Popescu \cite{HP}, 
and is sequentially Cohen-Macaulay. 
Moreover, since $J$ is squarefree, the simplicial complex associated with $\wS/J$ 
is {\it non-pure shellable} in the sense of Bj\"orner and Wachs \cite{BW}.

%Hence $\Theta$ also forms an $\Ext_{\wS}^i(\wS/J,\wS)$-regular 
%sequence for all $i$, and we have $\wS/(\Theta) \otimes_{\wS} \Ext_{\wS}^i(\wS/J,\wS) 
%\cong \Ext_S^i(S/I,S)$.   While the standard polarizations always satisfy this proerty, general polarizations fail it.     

Inspired by Kalai's theory on the {\it algebraic shifting} of simplicial complexes (c.f. \cite{K}), 
Aramova, Herzog and Hibi \cite{AHH2} introduced the operation sending  
a monomial $\m \in S$ of \eqref{alpha expression} to the squarefree monomial 
$$\m^\sigma:= \prod_{i=1}^e x_{\alpha_i+i-1}$$
in a polynomial ring $T:= \kk[x_1, \ldots, x_N]$ with  $N \gg 0$. 
If $I \subset S$ is a Borel fixed monomial ideal, we can define the squarefree monomial ideal 
$I^\sigma \subset T$ in the natural way (this construction works for general monomial ideals, 
but is important for Borel fixed ideals).  
This operation has the remarkable property that $\beta^S_{i,j}(I)=\beta_{i,j}^T(I^\sigma)$ 
for all $i,j$, as shown in \cite{AHH2}. Here $\beta_{i,j}(-)$ denotes the graded Betti number as usual.  

In Section~4, we will study $I^\sigma$ through our polarization $J:=\BoX(I)$. In fact, $\Theta_1:=\{ \, x_{i,j} -x_{i+1, j-1} \mid 1 \le i <n,   1< j \le d \, \}$ 
also forms a $\wS/J$-regular sequence and we have $\wS/(J+(\Theta_1)) \cong T/I^\sigma$ 
(if we set the number $N$ of the variables of $T$ to be $n+d-1$). 
Hence we get  a new proof of the equation $\beta^S_{i,j}(I)=\beta_{i,j}^T(I^\sigma)$. 
Moreover, we have $\beta_{i,j}^T(\Ext_T^k(T/I^\sigma, T)) =\beta_{i,j}^S(\Ext_S^k(S/I, S))$ 
for all $i,j,k$. Murai (\cite{M}) has generalized the operation $(-)^\sigma$ so that the equations on 
the Betti numbers remain true. We can also understand his operation using $\BoX$. 
In fact, it is enough to change a $\wS/J$-regular sequence $\Theta' \subset \wS$.

\section{Preparation}
We introduce the convention and notation used throughout the paper. 
Let $S=\kk[x_1, \ldots, x_n]$ be a polynomial ring over a field $\kk$. 
The $i^{\rm th}$ coordinate of $\ba \in \NN^n$ is denote by $a_i$ 
(i.e., we change the font). For $\ba \in \NN^n$, $x^\ba$ denotes 
the monomial $\prod_{i=1}^n x_i^{a_i} \in S$.  For a monomial $\m:=x^\ba$, set 
$\deg (\m) :=\sum_{i=1}^n a_i$ and $\deg_i (\m) :=a_i$. 
We define the order $\succeq$ on $\NN^n$ so that 
$\ba \succeq \bb$ if $\ a_i \ge b_i$ for all $i$. 
We refer \cite{BH,E} for  unexplained terminology. 

Take $\bd \in \NN^n$ with  $d_i \geq 1$ for all $i$, and  
set  $$\wS:=\kk[\, x_{i,j} \mid 1 \leq i \leq n, 1 \leq j \leq d_i \,].$$ 
Note that $$\Theta := \{ x_{i,1}-x_{i,j} \mid 1 \leq i \leq n, \, 2 \leq j \leq d_i \, \} \subset \wS$$ 
forms a regular sequence with $\wS/(\Theta) \cong S$. 
Here the isomorphism is induced by the ring homomorphism $\phi:\wS \to S$ with 
$\phi(x_{i,j})=x_i$. Throughout this paper, $\wS$ and $\Theta$ are used in this meaning, 
while the choice of $\bd \in \NN^n$ depends on the context.  

\begin{dfn}\label{pol def}
For a monomial ideal $I \subset S$, 
a {\it polarization} of $I$ is a squarefree monomial ideal $J \subset \wS$ satisfying 
the following conditions. 
\begin{itemize}
\item[(i)] Through the isomorphism $S \to \wS/(\Theta)$, we have $\wS/(\Theta) \otimes_{\wS} \wS/J \cong S/I$.  
\item[(ii)]  $\Theta$ forms a $\wS/J$-regular sequence. 
\end{itemize}
\end{dfn}

Clearly, the condition (i) holds if and only if $\phi(J)=I$.  
The following is a well-known fact, and a proof is found in \cite[Lemma~6.9]{NR}.  

\begin{lem}[c.f. {\cite[Lemma~6.9]{NR}}]\label{NR} 
Let $I$ and $J$ be monomial ideals of $S$ and $\wS$ respectively. 
Assume that the condition (i) of Definition~\ref{pol def} is satisfied. 
Then the condition (ii) is equivalent to the following.
\begin{itemize}
\item[(ii')] $\beta_{i,j}^{\wS}(J)=\beta_{i,j}^S(I)$ for all $i,j$. 
\end{itemize}
\end{lem}

While the proof in \cite{NR} concerns only the case $\# \Theta=1$, it works in the general case.  
If $\Theta$ does not form a $\wS/J$-regular sequence, the relation between $\beta_{i,j}^{\wS}(J)$ and 
$\beta_{i,j}^S(I)$ is not simple. So it is better to compare the Hilbert series of $\wS/J$ 
with that of $S/I$ (recall that the Hilbert series is determined by the Betti numbers.) 

\medskip

For a monomial $x^\ba$ with $\ba \preceq \bd$, set 
$$\pol(x^\ba):=\prod_{1 \le i \le n} 
x_{i,1} x_{i,2} \cdots x_{i, a_i} \in \wS.$$ 
Let $I \subset S$ be a monomial ideal 
with $\ba \preceq \bd$ for all $x^\ba \in G(I)$.  
Here $G(I)$ denotes the set of minimal (monomial) generators of $I$. 
Then it is well-known that 
$$\pol(I)=( \, \pol(x^\ba) \mid x^\ba \in G(I) \, )$$
gives a polarization of $I$, which is called the {\it standard polarization}. 
(If the reader is nervous about the choice of $\bd \in \NN^n$,  take it so that  $x^{\bd}$ is the least common multiple 
of the minimal generators of $I$. Anyway, for the properties considered in this paper, the choice of $\bd$ 
is not essential.) 
While all monomial ideals have the standard polarizations, some have alternative ones. 

\medskip 

Let $d$ be a positive integer, and set 
\begin{equation}\label{wS for box}
\wS:=\kk[ \, x_{i,j} \mid 1 \le i \le n, \, 1 \le j \le d \, ]. 
\end{equation}
For a monomial $x^\ba \in S$ with $e:=\deg  (x^\ba) \leq d$, set 
$b_i:=\sum_{j=1}^ia_j$ for each $i\geq 0$ (here $b_0=0$), and 
$$\BoX(x^\ba):=\prod_{\substack{1 \leq i \leq n \\b_{i-1}+1 \leq j \leq b_i }}x_{i,j} \in \wS.$$
If $a_i=0$ then $b_{i-1}=b_i$ and $x_{i,j}$ does not divide $\BoX(x^\ba)$ for all $j$. 
If $\m =x^\ba \in S$ is the monomial of \eqref{alpha expression}, 
then we have $b_i = \max\{ \, j \mid \alpha_j \le i \, \}$ and the above definition 
of $\BoX(x^\ba)$ coincides with the one given in \eqref{b-pol}.

Let $I \subset S$ be a monomial ideal with $\deg (x^\ba) \le d$ for all $x^\ba \in G(I)$.  Set 
$$\BoX(I):=( \, \BoX(x^\ba) \mid x^\ba \in G(I) \, ) \subset \wS.$$ 
Occasionally, this ideal gives a polarization of $I$. 
Note that the condition (i) of Definition~\ref{pol def} is always satisfied,  
and the problem is the condition (ii).  

In the sequel, when we treat $\BoX(I)$, we assume that $\wS$ is the one in \eqref{wS for box} and  
$\deg (\m) \le d$ for all $\m \in G(I)$.

\begin{exmp}\label{1st exm}
(1) For $I=(x^2, xy,  xz, y^2, yz) \subset \kk[x,y,z]$, 
we have 
$$\BoX(I)=(x_1x_2, x_1y_2,  x_1z_2, y_1y_2, y_1z_2),$$
and it gives a polarization. In fact, since $I$ is Borel fixed, we can use Theorem~\ref{main} below. 
It is essentially different from the standard polarization 
$$\pol(I)=(x_1x_2, x_1y_1,  x_1z_1, y_1y_2, y_1z_1).$$ 
More precisely, $\BoX(I)$ and $\pol(I)$ are different even after permutation of variables.   
%In fact, if $\m \in G(\pol(I))$ is divided by $y_i$ for some $i$, then $y_1$ divides $\m$. 
%However, $y_1$ does not divide $x_1y_2 \in G(\BoX(I))$ and 
%$y_2$ does not divide $y_1z_2 \in G(\BoX(I))$. 

(2) In general, $\BoX(I)$ does not give a polarization. 
For example, if 
$I=(xyz,x^2y,xy^2,x^3)$, then $\BoX(I)=(x_1y_2z_3, x_1x_2y_3,x_1y_2y_3,x_1x_2x_3)$,  
and it is not  a polarization. To see this, use Lemma~\ref{NR}. 
Note that $I$ is a {\it stable} monomial ideal, and Borel fixed ideals 
are nothing other than {\it strongly stable} monomial ideals (see \cite{AHH} for the definitions). 
\end{exmp}

\begin{dfn}\label{faithful def}
We say a polarization $J$ of $I$ is {\it faithful}, if $\Theta$ 
forms an $\Ext_{\wS}^i(\wS/J, \wS)$-regular sequence for all $i$. 
\end{dfn}

If  a polarization $J$ of $I$ is faithful, then we have 
$$ \wS/(\Theta) \otimes_{\wS} \Ext_{\wS}^i(\wS/J, \wS) \cong  \Ext_S^i(S/I, S).$$ 
In fact, the long exact sequences of $\Ext_{\wS}^\bullet(-, \wS)$ yield 
$$\wS/(\Theta) \otimes_{\wS} \Ext_{\wS}^i(\wS/J, \wS)  \cong \Ext_{\wS}^{i+(\#\Theta)}(\wS/(J+(\Theta)),\wS).$$
Since $\Theta \subset \wS$ forms a $\wS$-regular sequence with $\wS/(J+(\Theta)) \cong S/I$, 
we have 
$$\Ext_{\wS}^{i+(\#\Theta)}(\wS/(J+(\Theta)),\wS) \cong \Ext_S^i(S/I,S).$$

Hence, if $J$ is faithful,  $\Ext_S^i(S/I, S)$ and $\Ext_{\wS}^i(\wS/J, \wS)$ have  
the same degree  and Betti numbers. So $S/I$ and $\wS/J$ have the same arithmetic degree in this case. 

\begin{rem}
For any $I$, the standard polarization is always faithful by \cite[Corollary~4.10]{Sb} 
(see also \cite[Theorem~4.4]{Y}). It is an easy exercise to show that if $S/I$ is Cohen-Macaulay, then 
any polarization of $I$ is faithful. In Lemma~\ref{seqCM2} below, we will generalize this fact. 
\end{rem}

\begin{exmp}\label{non faithful}
For the ideal $I:=(x^2y, x^2z, xyz, xz^2, y^3, y^2z,yz^2)$ of $S:=\kk[x,y,z]$,  
$J:=\BoX(I) \subset \wS$ gives a polarization (to see this, compute the Betti numbers). 
However, $\deg \Ext_S^3(S/I,S)=6$ and 
$\deg \Ext^3_{\wS}(\wS/J,\wS)=5$. Hence $J$ is not faithful.    
\end{exmp}

%The following lemmata concerning the sequentially Cohen-Macaulay ideals (or modules) are easy. 
%So we omit the proofs here.  

Let $M$ be a finitely generated $S$-module. 
We say $M$ is {\it sequentially Cohen-Macaulay} if $\Ext_S^{n-i}(M,S)$ 
is either a Cohen-Macaulay module of dimension $i$ or the 0 module for all $i$. 
The original definition is given by the existence of a certain filtration 
(see \cite[III, Definition~2.9]{St}), 
however it is equivalent to the above one by \cite[III, Theorem~2.11]{St}.

\begin{lem}\label{seqCM}
Let $M$ be a sequentially Cohen-Macaulay $S$-module, and $y \in S$ a non-zero divisor of 
$M$. Then $y$ is a non-zero divisor of $\Ext_S^i(M,S)$ for all $i$, and  
$M/yM$ is a sequentially Cohen-Macaulay module with 
$$\Ext_S^{i+1}(M/yM,S) \cong \Ext_S^i(M,S)/y \cdot \Ext_S^i(M,S).$$ 
Moreover, we have 
$$\Ass(M/yM)=\{ \, \fp \mid \text{$\fp$ is a minimal prime of $\fp'+(y)$ 
for some $\fp' \in \Ass(M)$} \, \}.$$
If $y \in S_1$, and all associated primes of $M$ are generated by elements in $S_1$, then 
$$\Ass(M/yM)=\{ \, \fp'+(y) \mid \fp' \in \Ass(M) \, \}.$$
\end{lem}

To prove this lemma, recall the following basic properties of 
a finitely generated module $N$ over $S$ (c.f. \cite[Theorem~8.1.1]{BH}).  
\begin{itemize}
\item[(1)] $\dim_S (\Ext^i_S (N,S)) \leq n-i$ for all $i$. 
\item[(2)] For a prime ideal $\fp \subset S$ of codimension $c$, 
$\fp \in \Ass(N)$ if and only if $\fp$ is an associated (equivalently, minimal) 
prime of $\Ext_S^c(N,S)$. 
\end{itemize}

\begin{proof}
By the above remark, we have $\Ass (M)=\bigcup_i \Ass(\Ext_S^i(M,S))$. 
Hence the former half of the lemma is easy. 
To see the next assertion, let $\fp \subset S$ be a prime ideal of codimension $c$. 
Then we have;
\begin{eqnarray*}
\fp \in \Ass(M/yM) &\iff& 
\fp S_{\fp} \in \Ass_{S_{\fp}}(\Ext_S^c(M/yM,S) \otimes_S S_{\fp})\\
&\iff& \dim_{S_{\fp}} (\Ext_S^{c-1}(M,S) \otimes_S S_{\fp} )=n-c+1 \, \text{and} \, y \in \fp\\
&\iff& \text{$\exists$ $\fp' \in \Ass(\Ext_S^{c-1}(M,S))$ with $\codim \fp'=c-1$}\\ 
&\ &  \text{$\fp' \subset \fp$ and $y \in \fp$}\\
&\iff& \text{$\exists$ $\fp' \in \Ass(M)$ with $\codim \fp'=c-1$, $\fp' \subset \fp$ and $y \in \fp$}\\ 
&\iff& \text{$\exists$ $\fp' \in \Ass(M)$ such that $\fp$ is a minimal prime of $\fp' +(y)$.}
\end{eqnarray*}
The last assertion of the lemma is clear now, since $\fp'+(y)$ is a prime ideal for all $\fp' \in \Ass(M)$ 
in this case. 
\end{proof}

\begin{lem}\label{seqCM2}
Let $J$ be a polarization of $I$. 
If $\wS/J$ is sequentially Cohen-Macaulay, then so is $S/I$, and $J$ is faithful. 
\end{lem}

\begin{proof} 
Follows from the first assertion of Lemma~\ref{seqCM}. 
\end{proof}

\begin{rem}
Even if $S/I$ is sequentially Cohen-Macaulay, a polarization $J$ is not 
necessarily faithful. In fact, $S/I$ of Example~\ref{non faithful} is  sequentially Cohen-Macaulay.  
\end{rem}

\begin{dfn}\label{pretty clean}
Let $M$ be an $S$-module, and let 
$$\cF: 0=M_0 \subset M_1 \subset M_2 \subset \cdots \subset M_t =M$$ 
be a prime filtration, that is, there is a prime ideal $\fp_i$ such that  
$M_i/M_{i-1} \cong S/\fp_i$ for each $1 \le i \le t$. Herzog and Popescu (\cite{HP}) call   
the filtration $\cF$ is {\it pretty clean} if $i< j$ and $\fp_i \subseteq \fp_j$ imply $\fp_i=\fp_j$.  
\end{dfn}

For example, if $\codim \fp_i \ge \codim \fp_j$ for all $i, j$ with $i < j$, then $\cF$ is pretty clean. 
By \cite[Theorem~4.1 and Corollary~3.4]{HP}, if $M$ admits a pretty clean filtration $\cF$ 
then $M$ is sequentially Cohen-Macaulay and $\Ass M =\{ \, \fp_i \mid 1 \le i \le t \, \}$.  

\section{Main Results}
We say that a monomial ideal $I$ is {\it Borel fixed}, if $\m \in I$, $x_i|\m$ and $j <i$  
imply $(x_j/x_i) \cdot \m \in I$.  If $\chara(\kk) >0$, this terminology  
is unnatural (see \cite[\S 15.9.2]{E} for detail), and the terms  {\it 0-Borel fixed ideals} or {\it strongly stable monomial ideals} 
are also used in literature.  However, we just call it a Borel fixed ideal for simplicity.

For a monomial $\m \in S$, set 
$$\nu(\m):=\max \{ \, i \mid \text{$x_i$ divides $\m$} \,\}.$$
Similarly, for a monomial ideal $I \subset S$, set $\nu(I):=\max \{ \, \nu(\m) \mid \m \in G(I) \, \}$. 
If $I$ is Borel fixed, it is well know that 
$\nu(I)=\pd_S (S/I)$ (c.f. \cite[Corollary~15.25]{E}), while we do not use this fact.

\begin{lem}\label{gens of bpol}
If $I$ is a Borel fixed ideal (with $\deg (\m) \le d$ for all $\m \in G(I)$),  then 
$$\BoX(I)=( \, \BoX(\m) \mid \text{$\m \in I$ with $\deg(\m) \leq d$}\, ).$$
\end{lem}

\begin{proof}
Since the inclusion ``$\subseteq$" is clear, it suffices to show the converse. 
For the contrary, assume that there is some $\m \in I$ 
with $\deg (\m) \leq d$ and $\BoX(\m) \not \in \BoX(I)$. 
Take $\m$ so that it has the smallest degree among these monomials. 
It is clear that $\m \not \in G(I)$. 
Hence there is some $i$ with $x_i | \m$ and $\m':=\m/x_i \in I$. 
Set $l:= \nu(\m)$. Since $I$ is Borel fixed, we have 
$\m'':=\m/x_l= (x_i/x_l)\cdot \m' \in I$.
Since $\deg (\m'') < \deg (\m)=:e$, we have $\BoX(\m'') \in \BoX(I)$. 
Hence $\BoX(\m)= x_{l, e} \cdot \BoX(\m'') \in \BoX(I)$. This is a contradiction.  
\end{proof}

As shown in \cite[Proposition~5.2]{HP}, the quotient $S/I$ of a Borel fixed ideal $I$ has 
a pretty clean filtration. The next result states that the same is true for $J:=\BoX(I)$. 
Moreover, since $J$ is a radical ideal, $\wS/J$ actually admits a clean filtration by 
\cite[Corollary~3.5]{HP}. 
Hence the simplicial complex associated with $J$ 
is non-pure shellable. 

\begin{thm}\label{pre main}
Let $I$ be a Borel fixed ideal, and set $J:=\BoX(I)$. Then $\wS/J$ 
has a  pretty clean filtration, in particular, $\wS/J$ is sequentially Cohen-Macaulay.  
\end{thm}

\begin{proof}
Set $l:=\nu(I)$. Then $\{ \, \m \in G(I) \mid \nu(\m)=l \, \}$ is non-empty. 
Let $\m$ be the maximum element of this set with respect to the lexicographic order 
(not the {\it degree} lexicographic order). 
 If $\m=x_l$, then $I$ (resp. $J$) is a prime ideal 
$(x_1, \ldots, x_l)$ (resp. $(x_{1,1}, x_{2,1}\ldots, x_{l,1})$) and there is 
nothing to prove. So we may assume that $\m \ne x_l$, and set $\m_1:=\m/x_l$. 
Since $\m \in G(I)$, we have $\m_1\not \in I$. 

\bigskip

\noindent {\it Claim 1.} The ideal $I_1:= I + (\m_1)$ is Borel fixed.

\bigskip 

\noindent{\it Proof of Claim 1.} It suffices to show that $x_i | \m_1$ and $j < i$ imply 
$(x_j/x_i) \cdot \m_1 \in I$. 
Note that $\m':=x_l \cdot (x_j/x_i) \cdot \m_1 = (x_j/x_i) \cdot \m \in I$ and $\m'  > \m$ 
with respect to the lexicographic order. From our choice of $\m$, we have  $\m' \not \in G(I)$. 
Hence there is some $k$ such that $x_k | \m'$ and $\m'/x_k \in I$. 
If $k=l$, then we have $(x_j/x_i) \cdot \m_1 = \m'/x_k \in I$. 
So we may assume that $k \ne l$ and $\nu(\m'/x_k)=l$. Since $I$ is Borel fixed,  
we have $(x_j/x_i) \cdot \m_1 = \m'/x_l=(x_k/x_l) \cdot (\m'/x_k) \in I$. \qed  

\bigskip

If $\m_1= \prod_{i=1}^lx_i^{a_i}$, then  
$$\n:=\BoX(\m_1)=\prod_{\substack{1 \leq i \leq l \\b_{i-1}+1 \leq j \leq b_i }}x_{i,j},$$
where  $b_i:=\sum_{j=1}^ia_j$ for each $i\geq 0$ (here $b_0=0$). 
Note that $b_l=\deg (\m_1)=\deg (\n)$. 

\bigskip

\noindent{\it Claim 2.}  With the above notation, we have 
$J : \n = (x_{i, b_i+1} \mid 1 \leq i \leq l)$.  

\bigskip

\noindent{\it Proof of Claim 2.} First we prove that $x_{i, b_i+1} \cdot \n \in J$ 
for $1 \leq i \leq l$. Note that $x_i \cdot \m_1 = (x_i/x_l) \cdot \m \in I$. 
Since $\deg(x_i \cdot \m_1)=\deg (\m) \le d$, we have $\BoX(x_i \cdot \m_1) \in J$ by Lemma~\ref{gens of bpol}. 
If $\nu(\m_1)\leq i$, then we have $b_i=\deg (\n)$ and $x_{i, b_i+1} \cdot \n= \BoX(x_i \cdot \m_1) \in J$. 
Hence we may assume that $\nu(\m_1) >i$, and we can take  
$k:= \min \{ \, j \mid a_j >0, \, j >i \, \}$. 
Since $\m' := (x_i/x_k) \cdot \m_1$ is in $I$  by Claim~1, we have $\BoX(\m') \in J$  by Lemma~\ref{gens of bpol}. 
Hence $x_{i, b_i+1} \cdot \n = x_{k, b_i+1} \cdot \BoX(\m') \in J$.  

Next we prove $J : \n \subseteq (x_{i, b_i+1} \mid 1 \leq i \leq l)$. 
For the contrary, assume that there is a monomial $\n' \in \wS \setminus (x_{i, b_i+1} \mid 1 \leq i \leq l)$ satisfying $\n' \cdot \n \in J$. 
Then there is a monomial $\m'' =\prod x_i^{c_i} \in G(I)$ such that $\BoX(\m'')$ divides $\n' \cdot \n$.  
By the present assumption, we have that $\BoX(\m'')  \not \in (x_{i, b_i+1} \mid 1 \leq i \leq l)$. 
Under this assumption, we have the following. 

\bigskip 

\noindent{\it Claim 2.1.} Set $d_i :=\sum_{j=1}^i c_j$. Then $b_i \ge d_i$ for all $i$. 

\bigskip 

The above fact completes the proof of Claim 2. 
To see this, take the expression $\m_1 := \prod_{i=1}^e x_{\alpha_i}$ as \eqref{alpha expression}, 
where $e = \deg(\m_1)$. We have $e=b_l \ge d_l = \deg(\m'')=:f$.   
Moreover, since $I$ is Borel fixed and $b_i \ge d_i$ for all $i$, $\m'' \in I$ implies that 
$\prod_{i=1}^f x_{\alpha_i} \in I$. It follows that $\m_1 \in I$, which is a contradiction. 

\bigskip

\noindent{\it Proof of Claim 2.1.} 
Clearly, $b_0 = d_0=0$. 
Hence, if the claim does  not hold, there is some $i \ge 1$ such that $(b_i \ge ) \, b_{i-1} \ge d_{i-1}$ and $b_i < d_i$. 
Note that $x_{i,j}$ divides $\BoX(\m'')$ if and only if $d_{i-1}+1 \leq j \leq d_i$. 
Hence, under the present assumption,  $x_{i, b_i+1}$ divides $\BoX(\m'')$. 
This is a contradiction. \qed 

\bigskip

 \noindent{\it The continuation of the proof of Theorem~\ref{pre main}.} 
Set $J_1 := J +(\n)$ and $\fp:= (x_{i, b_i+1} \mid 1 \leq i \leq l)$. 
Then $J_1/J \cong (\wS/\fp)$ up to degree shift, and $\BoX(I_1)=J_1$.  
If $I_1$ is not a prime ideal, applying the above argument to $I_1$, 
we get a Borel fixed ideal $I_2 \, (\supset I_1)$ such that 
$\BoX(I_2)/J_1$ satisfies the similar property to 
$J_1/J$. Repeating this procedure, we have a sequence of Borel fixed ideals 
$$I=I_0 \subset I_1 \subset I_2 \subset \cdots \subset I_t$$
of $S$ such that $J_i:= \BoX(I_i)$ satisfies $J_i/J_{i-1} \cong \wS/\fp_i$ up to degree shift for all $i \geq 1$. 
Here $\fp_i \subset \wS$ is a prime ideal of the form 
$( \, x_{j, c_{i,j}} \mid 1 \leq j \leq l_i)$ for some 
$l_i, c_{i,j} \in \NN$. By the noetherian property of $S$,  
the procedure eventually terminates, that is, $I_t$ will become a prime ideal. 
In this case, $J_t=\BoX(I_t)$ is also a prime ideal, and we have a  prime filtration   
$$0 \subset  J_1/J \subset J_2/J \subset \cdots \subset J_t/J \subset \wS/J.$$
This is a pretty clean filtration. In fact, 
$\nu (I_1) \le \nu (I)$ by the construction. Similarly, 
$\nu (I_j) \le \nu (I_i)$ holds for all $i,j$ with $j \geq i$. On the other hand, we have 
$\codim \fp_i = l_i = \nu(I_i)$. Hence  
$\codim \fp_j \le \codim \fp_i$  for all $j \geq i$.  
Now recall the remark after Definition~\ref{pretty clean}.  
\end{proof}

\begin{rem}
By the above proof, we see that any associated prime of $J$ is of the form  
$( \, x_{i, c_i} \mid 1 \leq i \leq m)$ for some $m, c_i \in \NN$ with 
$c_1 \le c_2 \le \cdots \le c_m$.
\end{rem}
\begin{thm}\label{main}
If $I \subset S$ is a Borel fixed ideal, then $J:=\BoX(I)$ gives a polarization of $I$, 
which is faithful.     
\end{thm}

\begin{proof} To see that $J$ is a polarization, it suffices to show that $\Theta$ forms a $\wS/J$-regular sequence. 
So, assuming that a subset $\Theta'$ of $\Theta$ forms a $\wS/J$-regular sequence, we show that   
$\Theta' \cup \{ \, x_{i,1}-x_{i,j} \,\}$ is also a $\wS/J$-regular sequence for  $x_{i,1}-x_{i,j} \in \Theta \setminus \Theta'$. 
Since $\wS/J$ is sequentially Cohen-Macaulay and $\Theta'$ is assumed to be a regular sequence, 
$\wS/(J+(\Theta'))$  is also sequentially Cohen-Macaulay and 
$$\Ass_S(\wS/(J +(\Theta'))) = \{ \, \fp + (\Theta') \mid \fp \in \Ass (\wS/J) \, \}$$ 
by the repeated use of Lemma~\ref{seqCM}.  Since all $\fp \in \Ass(\wS/J)$ is of the form 
$(x_{k, c_k} \mid 1 \leq k \leq m)$, $x_{i,1}-x_{i,j}$ is $\wS/(J +(\Theta'))$-regular. 

The faithful-ness follows from Lemma~\ref{seqCM2}.  
\end{proof}

S. Murai told us that Theorem~\ref{main} can be shown by using his \cite[Proposition~1.9]{M}. 
We will explain this idea in Remark~\ref{Murai} below, since it requires  {\it (generalized) 
squarefree operations} introduced in the next section. 
%(For general $\{a_i\}_{i \in \NN}$, $\alpha^a$ has no relation to polarization. 
%Our choice of $\{a_i\}_{i \in \NN}$ makes this operator ``polarization-like".) 

However, this second proof does not give a pretty clean filtration of $\wS/\BoX(I)$ 
(equivalently, the non-pure shellability of the associated simplicial complex) and the following generalization of Theorem~\ref{main}. 
Moreover, in the next section, we will show a new proof of \cite[Proposition~1.9]{M} using $\BoX(I)$, and 
gives a new perspective to the squarefree operations.

\begin{thm}\label{mixed}
Let $A$ be a subset of  $\{1,2, \ldots, n\}$. 
For a monomial $\m=x^\ba \in S$, set $\m_A:= \prod_{i \in A} x_i^{a_i}$, 
$\m_{-A}:=\prod_{i \not \in A} x_i^{a_i}$ and 
$$\BoX_A(\m):= \BoX(\m_A) \cdot \pol(\m_{-A}) \in \wS$$
(we set $\wS:=\kk[\, x_{i,j} \mid 1 \leq i \leq n, 1 \leq j \leq d \,],$  
where $d:=\max \{ \, \deg (\m) \mid \m \in G(I) \, \}$).   
If $I$ is Borel fixed, then $\wS/\BoX_A(I)$ has a pretty clean filtration, where  
$$\BoX_A(I):=( \, \BoX_A(\m) \mid \m \in G(I) \, ).$$ 
Moreover, $\BoX_A(I)$ gives a faithful polarization of $I$. 
\end{thm}

By the above theorem, we see that Borel fixed ideals have many alternative polarizations.  

\begin{lem}
In the situation of Theorem~\ref{mixed}, we have  
$$\BoX_A(I)=( \, \BoX_A(\m) \mid \text{$\m \in I$ with $\deg(\m) \leq d$}\, ).$$
\end{lem}

Clearly, this is a generalization of Lemma~\ref{gens of bpol}. 

\begin{proof}
It suffices to prove  ``$\supseteq$". Set $J:= \BoX_A(I)$. 
For the contrary, assume that there is some $\m=x^\ba \in I$ 
with $\deg (\m) \leq d$ and $\BoX_A(\m) \not \in J$. Since $\m \not \in G(I)$,  
there is some $i$ with $x_i | \m$ and $\m':=\m/x_i \in I$. 
If $i \not \in A$, then it is easy to see that $\BoX_A(\m)=x_{i, a_i}  \cdot \BoX_A(\m') \in J$. 
Hence we have $i \in A$.  If we replace $\nu(\m)$ by $\nu_A(\m):= \max \{ \, i \in A \mid a_i > 0 \, \}$, 
the last part of the proof of Lemma~\ref{gens of bpol} works verbatim, 
except that $\BoX_A(\m)= x_{\nu_A(\m), f} \cdot \BoX_A(\m'')$ with $f:=\sum_{i \in A}a_i$. 
\end{proof}

\noindent{\it Proof of Theorem~\ref{mixed}.}
For the former assertion, we imitate the proof of Theorem~\ref{pre main}. 
First,  take the same $\m \in \wS$ as in the proof of Theorem~\ref{pre main}   
(here $\nu(\m)= \nu (I)=:l$, and $\nu_A(I)$ is {\it not} used). 
As shown in Claim 1 of the original proof, $I+(\m_1)$ is Borel fixed. 

For the statement corresponding to Claim~2, we need modification. 
If $\m \ne x_l$, set $\m_1:= \m/x_l= \prod_{i=1}^n x^{a_i}$ and $\n=\BoX_A(\m_1)$. 
For each $i \in A$, set 
$$b_i := \sum_{j \in A, \, j \le i} a_j.$$ 
Next, we will show that $J : \n = \fp$, where 
$$\fp :=(x_{i, b_i+1} \mid i \in A, i \leq l)+(x_{i,a_i+1} \mid i \not \in A, i \le l).$$  

Note that $x_i \cdot \m_1 =(x_i/x_l) \cdot \m \in I$ for $i \le l$. 
If $i \not \in A$, then we have $x_{i, a_i+1} \cdot \n =\BoX_A(x_i \cdot \m_1) \in J$. 
If $i \in A$, then we can show that $x_{i, b_i+1} \cdot \n \in J$ 
by a similar argument to the proof of Claim~2, while  
we have to replace $\min \{ \, j \mid a_j >0, \, j >i \, \}$ by  
$\min \{ \, j \in A \mid a_j >0, \, j >i \, \}$.  Hence we have $J:\n \supset \fp$. 

To prove the converse, assume that a monomial $\n' \in \wS$ satisfies $\n' \cdot \n \in J$. 
Then there is a monomial $\m'' \in G(I)$ such that 
$\BoX_A(\m'')$ divides $\n' \cdot \n$.  
If $\n' \not \in (x_{i, a_i+1} \mid i \not \in A, \,  i \leq l)$, then 
$\BoX_A(\m'')  \not \in (x_{i, a_i+1} \mid i \not \in A, i \leq l)$ also. 
It means that $\deg_i (\m'') \leq a_i= \deg_i (\m_1)$ for all $i \not \in A$. 
Now, concentrating our attention to the variables $x_i$ with $i \in A$ and $i \le l$, 
we can use the proof of Claim~2 (almost) verbatim, and we see that the assumption $\n' \not \in \fp$ 
implies that $\m_1 \in I$. This is a contradiction. 

Hence we have $J : \n = \fp$, and a pretty clean filtration can be constructed as in 
(the final step of) the proof of Theorem~\ref{pre main}. 

The above argument shows that any associated prime of $\wS/\BoX(I)$ is of the form 
$( \, x_{i, c_i} \mid 1 \leq i \leq m)$ (but we lost the relation $c_1 \le c_2 \le \cdots \le c_m$ here).  
Hence, by a similar argument to the proof of Theorem~\ref{main}, 
we can show that $J$ is a faithful polarization. 
%(since the argument is almost same, we omit this part of the proof). 
\qed

\section{Application to squarefree operation}
Throughout this section, let $\{a_i\}_{i \in \NN}$ be a non-decreasing sequence of non-negative integers. 
We also assume that $a_0=0$ for the convenience.   

Let $T=\kk[x_1, \ldots, x_N]$ be a polynomial ring with $N \gg 0$. 
For a monomial $\m \in S=\kk[x_1, \ldots, x_n]$, take the expression 
$\m= \prod_{i=1}^e x_{\alpha_i}$ as \eqref{alpha expression}.  
Murai (\cite{M}) defined the operation $(-)^{\sigma(a)}$ which sends  $\m$ to   
$$\m^{\sigma(a)}:= \prod_{i=1}^e x_{\alpha_i +a_{i-1}} \in T.$$
For a monomial ideal $I \subset S$, he also set 
$$I^{\sigma(a)}:= ( \, \m^{\sigma(a)} \mid \m \in G(I) \, ) \subset T.$$
(In \cite{M}, the symbol ``$\alpha^a$" is used for this operation. 
However, we change the notation, since the letter $\alpha$ has been used already.)

If $a_{i+1} > a_i$ for all $i$, then $\m^{\sigma(a)}$ is a squarefree monomial. 
In particular, if $a_i=i$ for all $i$, then $(-)^{\sigma(a)}$ coincides  
with the {\it squarefree operation} $(-)^{\sigma}$, which plays an important role in the 
construction of the  {\it symmetric shifting} of a simplicial complex 
(see \cite{AHH2}, and also \cite{K} for the original form of the shifting theory).

Let $L_a$ be the linear subspace of $S_1$ spanned by 
$$X_a := \{ \, x_{i,j} -x_{i', j'} \mid i+a_{j-1}=i'+a_{j'-1}  \, \},$$
and take a subset $\Theta_a \subset X_a$ so that it forms a basis of $L_a$. 
For example, if $a_i=i$ for all $i$, then we can take 
$$\{ \, x_{i,j} -x_{i+1, j-1} \mid 1 \le i <n,   1< j \le d \, \}$$
as $\Theta_a$.   
Clearly,  $\Theta_a$ is a $\wS$-regular sequence, and the ring homomorphism 
$\psi : \wS \to T \, (=\kk[x_1, \ldots, x_N])$ 
defined by $\wS \ni x_{i,j} \mapsto x_{i+a_{j-1}} \in T$ induces the isomorphism 
$\wS/(\Theta_a) \cong T$ (if we adjust the number $N$).   

\begin{prop}\label{Kalai-Murai}
Let $I \subset S$ be a Borel fixed ideal, and set $J:=\BoX(I)$. 
Then $\Theta_a$ forms a $\wS/J$-regular sequence, and we have 
$\wS/(\Theta_a)\otimes_{\wS} \wS/J  \cong T/I^{\sigma(a)}$ through the isomorphism $S/(\Theta_a) \to T$ 
(that is,  we have $\psi(J)=I^{\sigma(a)}$).  
\end{prop}

\begin{proof}
The latter assertion is clear by the expression \eqref{b-pol},  
and it suffices to prove the former. 
Recall that $\wS/J$ is sequentially Cohen-Macaulay, 
and any associated prime of $\wS/J$ is 
of the form  $( \, x_{i, c_i} \mid 1 \leq i \leq m\, )$ with 
$c_1 \le c_2 \le \cdots \le c_m$. 
If $x_{i,j}-x_{i',j'} \in \Theta_a$ and $i < i'$, then $a_{j-1}-a_{j'-1} = i'-i >0$. 
Since $\{ a_k \}_{k \in \NN}$ is a non-decreasing sequence, we have $j > j'$.    
Hence $\Theta_a$ forms a $\wS/J$-regular sequence 
by the same argument as the proof of Theorem~\ref{main}. 
\end{proof}

\begin{cor}[Murai, {\cite[Proposition~1.9]{M}}]\label{Murai Betti}
Let $I$ be a Borel fixed ideal. Then,   
$$\beta^S_{i,j}(I)=\beta^T_{i,j}(I^{\sigma(a)})$$ 
for all $i,j$.
\end{cor} 

\begin{proof}
The left (resp. right) side of the equation equals to $\beta_{i,j}^{\wS}(J)$ 
by Theorem~\ref{main} (resp. Proposition~\ref{Kalai-Murai}). 
\end{proof}

The original proof in \cite{M} uses a formula given in \cite{AHH}, and is very different from ours. 

\begin{rem}\label{Murai}
Murai told us that Corollary~\ref{Murai Betti} (i.e., his \cite[Proposition~1.9]{M}) 
can be used to prove  Theorem~\ref{main}. 
In fact, if $a_i = i \cdot n$ for each $i$, $(-)^{\sigma(a)}$ corresponds to our $\BoX$. 
To see this, assign our variable $x_{i,j}$ to his $x_{(j-1)\cdot n+i}$.  
Since $(-)^{\sigma(a)}$ preserves the Betti numbers of a Borel fixed ideal, 
it gives a polarization by Lemma~\ref{NR}. 
%(For general $\{a_i\}_{i \in \NN}$, $\alpha^a$ has no relation to polarization. 
%Our choice of $\{a_i\}_{i \in \NN}$ makes this operator ``polarization-like".) 
However, our proof has advantages as mentioned before Theorem~\ref{mixed}, and 
we can refine Corollary~\ref{Murai Betti} as follows.  
That is, Theorem~\ref{main} (the polarization $\BoX(I)$) and Corollary~\ref{Murai Betti} 
(generalized squarefree operation) imply each other, but our analysis of $\BoX(I)$ 
contains more precise information.  
\end{rem}

\begin{cor}
With the situation of Proposition~\ref{Kalai-Murai}, 
$\Theta_a$ forms an $\Ext_{\wS}^i(\wS/J, \wS)$-regular sequence for all $i$, and 
$$\wS/(\Theta_a) \otimes_{\wS} \Ext_{\wS}^i(\wS/J, \wS)  \cong \Ext_T^i(T/I^{\sigma(a)}, T).$$ 
Hence we have $$\beta_{i,j}^T(\Ext_T^k(T/I^{\sigma(a)}, T)) 
=\beta_{i,j}^S(\Ext_S^k(S/I, S))$$ for all $i,j,k$.  Similarly, $\deg (\Ext_T^i(T/I^{\sigma(a)}, T)) 
=\deg(\Ext_S^i(S/I, S))$ for all $i$, and hence $S/I$ and $T/I^{\sigma(a)}$ have the same arithmetic degree.
\end{cor}

\begin{proof}
Since $\Theta_a$ is a $\wS/J$-regular sequence and $\wS/J$ is sequentially Cohen-Macaulay, 
the former assertion follows form iterated use of Lemma~\ref{seqCM} 
(see also the argument after Definition~\ref{faithful def}).  
The equation on the Betti numbers holds, since the both sides equal to 
$\beta_{i,j}^{\wS}(\Ext_{\wS}^k(\wS/J, \wS))$. 
The  equations on the degrees can be proved in a similar way. 
\end{proof}

\begin{prop}
If $I \subset S$ is a Borel fixed ideal, then $T/I^{\sigma(a)}$ has a pretty clean filtration.   
In particular, if $I^{\sigma(a)}$ is squarefree (e.g., if $a_{i+1} > a_i$ for all $i$), 
the corresponding simplicial complex of $T/I^{\sigma(a)}$ is non-pure shellable.  
\end{prop}

\begin{proof}
Take the pretty clean filtration 
$0 \subset  J_1/J \subset J_2/J \subset \cdots \subset J_t/J \subset \wS/J$ ($J_0=J$) 
constructed in the proof of Theorem~\ref{pre main}. 
Recall that  $J_i/J_{i-1} \cong \wS/\fp_i$ up to degree shift for each $i \geq 1$. 
Since $\fp_i \in \Ass(\wS/J)$, $\Theta_a$ forms a $\wS/\fp_i$-regular sequence 
by the same argument as in the proof of Proposition~\ref{Kalai-Murai}.  
Moreover,  $\wS/(\Theta_a) \otimes_{\wS} \wS/\fp_i \cong T/\fq_i$ for some prime ideal $\fq_i \subset T$ with 
$\codim \fp_i =\codim \fq_i$.  From the exact sequence 
$0 \to J_{i-1}/J \to J_i/J \to \wS/\fp_i \to 0$, we have the exact sequence 
$$0 \to \wS/(\Theta_a) \otimes_{\wS} J_{i-1}/J \to \wS/(\Theta_a) \otimes_{\wS} J_i/J \to 
\wS/(\Theta_a) \otimes_{\wS} \wS/\fp_i \to 0$$
by \cite[Proposition~1.1.4]{BH}. Set $M_i:= \wS/(\Theta_a) \otimes_{\wS} J_i/J_{i-1}$. Then 
$0 \subset M_1 \subset \cdots \subset M_t \subset T/I^{\sigma(a)}$ is a pretty clean 
filtration.  
\end{proof}

\section*{Acknowledgment}
The study of this subject started when I visited the Mathematical Institute of 
the University of Bergen. I am grateful for the warm hospitality. 
I thank Professor Gunnar Fl\o ystad and Doctor Henning Lohne for stimulating discussion. 
I also thank Professor Satoshi Murai for telling me the relation to his work \cite{M}.  
Finally,  I wish to thank  anonymous referee for his/her careful reading of this manuscript. 

\end{document}